\documentclass[12pt]{amsart}

\usepackage{amssymb}
\usepackage{bm}
\usepackage{graphicx}
\usepackage{enumerate}
\usepackage{psfrag}
\usepackage{color}
\usepackage{url}

\newcommand{\excise}[1]{}

\newtheorem{thm}{Theorem}[section]

\newtheorem{cor}[thm]{Corollary}
\newtheorem{prop}[thm]{Proposition}
\newtheorem{conj}[thm]{Conjecture}

\newtheorem{prob}[thm]{Problem}

\theoremstyle{definition}
\newtheorem{example}[thm]{Example}
\newtheorem{remark}[thm]{Remark}
\newtheorem{defn}[thm]{Definition}

\numberwithin{equation}{section}

\newcommand{\ring}[1]{\ensuremath{\mathbb{#1}}}

\renewcommand\>{\rangle}
\newcommand\<{\langle}

\newcommand\NN{\ring{N}}

\newcommand\QQ{\ring{Q}}

\newcommand\ZZ{\ring{Z}}

\newcommand\bb{{\mathbf b}}

\renewcommand\aa{{\mathbf a}}
\renewcommand\bb{{\mathbf b}}
\newcommand\ee{{\mathbf e}}

\newcommand\iso{\cong}

\DeclareMathOperator\lcm{lcm} 
\DeclareMathOperator\bul{bul} 

\DeclareMathOperator\Betti{Betti} 

\begin{document}

\mbox{}
\title{Factorization invariants in numerical monoids \qquad}
\author{Christopher O'Neill}
\address{Mathematics Department\\Texas A\&M University\\College Station, TX 77840}
\email{coneill@math.tamu.edu}
\author{Roberto Pelayo}
\address{Mathematics Department\\University of Hawai`i at Hilo\\Hilo, HI 96720}
\email{robertop@hawaii.edu}

\date{\today}

\begin{abstract}
\hspace{-2.05032pt}
Nonunique factorization in commutative monoids is often studied using factorization invariants, which assign to each monoid element a quantity determined by the factorization structure.  For numerical monoids (co-finite, additive submonoids of $\NN$), several factorization invariants have received much attention in the recent literature.  In this survey article, we give an overview of the length set, elasticity, delta set, $\omega$-primality, and catenary degree invariants in the setting of numerical monoids.  For each invariant, we present current major results in the literature and identify the primary open questions that remain.  
\end{abstract}
\maketitle


\section{Motivating numerical monoids}\label{s:intro}

Factorization theory has its genesis in studying the decomposition of natural numbers into its prime (irreducible) divisors.  In this multiplicative monoid, the Fundamental Theorem of Arithmetic guarantees that such a decomposition is unique.  Many other monoids, however, have elements that fail to admit a unique factorization into irreducibles; in such cases, algebraists are frequently interested in measuring, in various ways, how different the plural factorizations of an element can be.  

One does not have to look far for examples.  In fact, certain submonoids of the natural numbers have a rich non-unique factorization theory.  \emph{Numerical monoids} (co-finite, additive submonoids of the natural numbers) have garnered much recent interest, as their factorization theories are complicated enough to be interesting while controlled enough to be tractable.  The goal of this paper is to introduce various perspectives and tools for studying factorizations in numerical monoids, and to present open problems that are of primary interest in this field.  

Several natural examples of numerical monoids are found in the literature.  Historically, Frobenius' coin-exchange problem asked what monetary values one could make with relatively prime coin denominations.  A more endearing manifestation of numerical monoids is the McNugget monoid, which enumerates the number of chicken McNuggets one could buy by using the standard package sizes of 6, 9, and 20.  As with the coin-problem and the McNugget monoid, all numerical monoids can be described using a finite set of generators.  We provide the formal definition below.  In what follows, we let $\NN$ denote the set of non-negative integers.

\begin{defn}\label{d:numericalmonoid}
Let $n_1, n_2, \ldots, n_k$ be a collection of relatively prime positive integers.  The \emph{numerical monoid} $S$ generated by these integers is the set of all non-negative, integral linear combinations of these numbers.  In other words, $$S = \<n_1, n_2, \ldots, n_k\> = \{a_1n_1 + a_2n_2 + \cdots + a_kn_k : a_i \in \NN\}.$$
\end{defn}

\begin{remark}\label{r:mingenset}
This set is an additive monoid with identity $0$, and the condition that $\gcd(n_1, n_2, \ldots, n_k) = 1$ guarantees that there exist only finitely many natural numbers that are not elements of $S$.  In fact, every co-finite additive submonoid of $\NN$ can be generated by a finite set of relatively prime natural numbers; thus, numerical monoids are precisely the co-finite, additive submonoids of $\NN$.  

While generating sets of a fixed numerical monoid are non-unique, there always exists a unique minimal generating set (with respect to set-theoretic containment).  As these minimal generators cannot be represented in terms of the other generators, they are precisely the \emph{irreducible} elements of the numerical monoid.  For the remainder of the paper, we will assume that $\{n_1, n_2, \ldots, n_k\}$ constitutes the unique minimal generating set for a numerical monoid $S$ and that $n_1 < n_2 < \ldots < n_k$.  
\end{remark}

Both Frobenius' coin problem and the McNugget problem ask which natural numbers arise as non-negative, integral linear combinations of their respective generators.  As such, understanding the largest natural number that is not an element of $S$ (named for the aforementioned coin problem) is often useful.

\begin{defn}\label{d:frobenius}
For a numerical monoid $S$, the values in $\NN \setminus S$ are called the \emph{gaps} of $S$, and $F(S) = \max(\NN \setminus S)$ is called the \emph{Frobenius number} of $S$. 
\end{defn}

\begin{example}\label{e:numericalmonoid}
For two relatively prime positive integers $n_1, n_2$, the numerical monoid $S = \<n_1, n_2\>$ has $F(S) = n_1n_2 - (n_1 + n_2)$ as its Frobenius number~\cite{frobsylvester}.  Furthermore, all but finitely many elements of $S$ can be written in multiple ways as non-negative integral combinations of $n_1$ and $n_2$.  The simplest example is the element $n_1n_2 \in S$, which can be seen as $n_1$ copies of $n_2$ or as $n_2$ copies of $n_1$.  These two representations of $n_1n_2$ as an element of $S$ are called \emph{factorizations}; see Section~\ref{s:factorization} for precise definitions and notation.  

The McNugget monoid $M = \<6,9,20\>$ has Frobenius number 43; thus, for every $n > 43$, one can purchase exactly $n$ McNuggets using the standard package sizes of $6,9$, and $20$.  As with the numerical monoid $\<n_1,n_2\>$, almost every element in $M$ has multiple distinct factorizations.  For example, $60$ can be expressed as three copies of $20$, ten copies of $6$, or four copies of $6$ plus four copies of $9$.  Each such expression corresponds to a way to purchase exactly 60 McNuggets.
\end{example}

The existence of a Frobenius number for a numerical monoid $S$ guarantees that past some value, every natural number lies in $S$.  Therefore, as a set, a numerical monoid $S$ is a scattering of ``small" elements (those less than $F(S)$) and an infinite ray after $F(S)$.  What is less understood is the factorization theory of these elements, which is hinted at in Example~\ref{e:numericalmonoid}.  In general, the factorization theory of numerical monoids with $2$ minimal generators is well-understood, while those numerical monoids with 3 or more irreducible elements have a much more complicated factorization structure with numerous open questions.  

Factorization structures are often classified and quantified using a variety of important \emph{factorization invariants}, which generally measure how far from unique a monoid element's factorizations are.  The coming sections introduce several standard invariants, together with results and open questions in the context of numerical monoids.  Section~\ref{s:factorization} describes the strongest (but most cumbersome) invariant, the set of factorizations.  Then, Section~\ref{s:lengthsets} describes the length set and elasticity invariants, which measure, in different ways, how many irreducible elements appear in the factorizations of an element.  Section~\ref{s:deltasets} gives several important results for delta sets, which measure the gaps in lengths that are possible between distinct factorizations of a single element.  While length sets, elasticity, and delta sets record only information about the length of factorizations of an element, the catenary degree, which is described in Section~\ref{s:catenary}, measures how many irreducible elements different factorizations of the same element do not share in common.  Lastly, as non-unique factorizations coincide with the presence of irreducible elements that are not prime, Section~\ref{s:omega} describes $\omega$-primality, which measures how ``far from prime" a monoid element is.  


\section{Sets of factorizations}\label{s:factorization}

\begin{defn}\label{d:factorizations}
Fix a numerical monoid $S = \<n_1, \ldots, n_k\>$ and an element $n \in S$.  A~\emph{factorization} of $n$ is an expression $n = u_1 + \cdots + u_r$ of $n$ as a sum of irreducible elements $u_1, \ldots, u_r$ of $S$.  Collecting like terms, we often write factorizations of $n$ in the form $n = a_1n_1 + \cdots + a_kn_k$ (see Remark~\ref{r:mingenset}).  Write 
$$\mathsf Z_S(n) = \{(a_1, \ldots, a_k) : n = a_1n_1 + \cdots + a_kn_k\} \subset \NN^k$$
for the \emph{set of factorizations} of $n \in S$.  When there is no ambiguity, we often omit the subscript and simply write $\mathsf Z(n)$.  
\end{defn}

\begin{remark}\label{r:ehrhart}
The notation for factorizations in Definition~\ref{d:factorizations} is motivated in part by connections to discrete geometry.  In particular, each factorization of $n$ corresponds to an integer point in $\NN^k$, and the set $\mathsf Z_S(n)$ forms a codimension-1 linear cross-section of $\NN^k$ that coincides with the set of integer points of a convex integral polytope.  This viewpoint allows the use of techniques from lattice point enumeration and other tools from combinatorial commutative algebra when examining the factorization structure of numerical monoids \cite{continuousdiscretely,cca}.  
\end{remark}

The set of factorizations (Definition~\ref{d:setoffactorizations}) is an example of a perfect factorization invariant in that it encapsulates enough information to uniquely determine the underlying monoid (Theorem~\ref{t:setoffactorizations}).  However, such complete information comes at a cost: extracting information from it (or even simply writing it down) is a nontrivial task.  Invariants derived from the set of factorizations, such as the length set and elasticity invariants (Section~\ref{s:lengthsets}), are more manageable and easier to work with, but necessitate a loss of information.  Below, $2^{\NN^k}$ denotes the power set of $\NN^k$.

\begin{defn}\label{d:setoffactorizations}
Given a numerical monoid $S = \<n_1, \ldots, n_k\>$, let
$$\mathcal Z(S) = \{\mathsf Z_S(n) : n \in S\} \subset 2^{\NN^k}$$
denote the \emph{set of factorizations} of $S$.  
\end{defn}

Theorem~\ref{t:setoffactorizations} follows from the fact that a minimal presentation of a numerical monoid $S$ (as well as its defining congruence) can be recovered from $\mathcal Z(S)$.  Example~\ref{e:setoffactorizations} exhibits one possible method of doing so; see~\cite{numerical,fingenmon} for more background on monoid congruences and minimal presentations.  In fact, for a fixed $n \gg 0$, the set $\mathsf Z_S(n)$ alone contains enough information to recover the monoid structure of $S$.  As such, it is perhaps not surprising that $\mathcal Z(S)$ is often hard to describe in complete detail.  

\begin{thm}\label{t:setoffactorizations}
For two numerical monoids $S$ and $S'$, the following are equivalent.  
\begin{enumerate}[(a)]
\item $S = S'$, 
\item $S \iso S'$, and
\item $\mathcal Z(S) = \mathcal Z(S')$.
\end{enumerate}
\end{thm}

\begin{proof}
\cite[Theorem~1.1]{fingenmon}.  
\end{proof}

\begin{example}\label{e:setoffactorizations}
Fix a numerical monoid $S$, and suppose that the two-element sets $\{(3,0,0),(0,2,0)\}$ and $\{(10,0,0),(0,0,3)\}$ are each a subset of some element of $\mathcal Z(S)$.  Since each is a subset of $\NN^3$, $S$ must have three minimal generators, say $S = \<n_1, n_2, n_3\>$.  Moreover, since $(3,0,0)$ and $(0,2,0)$ are factorizations of the same element, we must have $3n_1 = 2n_2$, and similarly that $10n_1 = 3n_3$.  Since $\gcd(n_1, n_2, n_3) = 1$, we conclude that $S = \<6,9,20\>$, the McNugget monoid from Example~\ref{e:numericalmonoid}.  
\end{example}

Larger elements in a numerical monoid $S$ are more likely to have larger factorization sets.  More precisely, if $S$ has $k$ minimal generators, the size of $\mathsf Z(n)$ grows on the order of $n^{k-1}$.  Theorem~\ref{t:factorasymp}, which specializes the asymptotic result in \cite{factorasymp} to numerical monoids, can be found in~\cite{hilbertfactor}.

\begin{thm}[{\cite[Theorem 3.9]{hilbertfactor}}]\label{t:factorasymp}
Given $S = \<n_1, \ldots, n_k\>$ numerical, there exist periodic functions $a_0, \ldots, a_{k-2}:\NN \to \QQ$ whose periods divide $\lcm(n_1, \ldots, n_k)$ such that
$$|\mathsf Z_S(n)| = \frac{1}{(k-1)! n_1 \cdots n_k}n^{k-1} + a_{k-2}(n)n^{k-2} + \cdots + a_1(n)n + a_0(n)$$
for all $n \in S$.  
\end{thm}

\begin{remark}\label{r:computefactor}
Much of the recent literature on numerical monoids has focused on algorithms for explicit computation.  It is worth noting that although Theorem~\ref{t:factorasymp} implies $|\mathsf Z(n)|$ is polynomial in $n$, is it $NP$-hard to compute this set in general, as doing so requires solving a system of Diophantine equations.  See \cite{compoverview} for a survey of recent computational advances in this area.  

Several of the factorization invariants introduced in the coming sections, including those derived directly from the set of factorizations, benefit from faster computation methods.  In particular, several of these invariants admit algorithms that do not require computing large sets of factorizations; see, for instance, Remarks~\ref{r:deltadynamic} and~\ref{r:omegadynamic}.  Additional computational traction can also be obtained by specializing to certain families of monoids, such as those with a fixed number of generators or generating sets of a certain form; see, for instance, Theorems~\ref{t:elastarith}, \ref{t:deltaarithmetic} and~\ref{t:catnaryarithmetic}.  

Another computation technique in computing invariant values is to first compute the Ap\'ery set of $S$ (a set determined by the complement of $S$; see \cite{numerical}).  Although computing the Ap\'ery set from a list of generators is also $NP$-hard in general, many invariants can then be computed directly from it.  For instance, computing the Frobenius number of a numerical monoid is $NP$-hard in general \cite{diophantinefrob}, but if the Ap\'ery set is known, its computation is trivial \cite{selmersformula}.  In this way, the complexity of computing an invariant (or several invariants) can be ``transferred'' to the computation of the Apery set \cite{compoverview,compapery,aperyhilbert}.  Additionally, for some special families of numerical monoids, the Apery set is known to admit a closed form; see \cite{ros99,arfnumerical}.  
\end{remark}

\section{Length sets and elasticity}\label{s:lengthsets}

While the set of factorizations for a numerical monoid is a perfect invariant, it is often cumbersome to compute and encode.  The next two sections focus on invariants obtained by passing from factorizations to their lengths.  In this section, we introduce length sets and elasticity as useful invariants for numerical monoids.  We begin with the definition of a length set.

\begin{defn}\label{d:lengthset}
Fix a numerical monoid $S = \<n_1, \ldots, n_k\>$, and fix $n \in S$.  Given a factorization $\aa = (a_1,a_2, \ldots, a_k) \in \mathsf Z_S(n)$, we denote by $|\aa|$ the number of irreducibles in the factorization $\aa$; that is, $|\aa| =  a_1 + \cdots + a_k$.  The set of factorization lengths of~$n$, denoted $\mathsf L_S(n) = \{|\aa| : \aa \in \mathsf Z_S(n)\}$, is called the \emph{length set} of $n$.  When there is no ambiguity, we often omit the subscript and simply write $\mathsf L(n)$.  
\end{defn}

\begin{example}\label{e:lengthset}
The length set of an element is easily computed from its set of factorizations.  In the numerical monoid $S = \<7,10,13\>$, the element $20$ has factorization set $\mathsf Z(20) = \{(1,0,1), (0,2,0)\}$ and length set $\mathsf L(20) = \{2\}$.  This length set, however, is shared with other elements (e.g., $\mathsf L(14) = \{2\}$).  This example highlights one manner in which information is frequently lost when passing from $\mathsf Z(n)$ to $\mathsf L(n)$.  
\end{example}

Although distinct elements in a numerical monoid can have identical length sets, one may ask whether distinct numerical monoids can have different sets of length sets. In what follows, $2^{\NN}$ denotes the power set of $\NN$.

\begin{defn}\label{d:setsoflengthsets}
For a numerical monoid $S$, its \emph{set of length sets} is given by 
$$\mathcal L(S) = \{\mathsf L(n) : n \in S\} \subset 2^\NN.$$
\end{defn}

Investigations into whether sets of length sets determine finitely generated, cancellative, commutative monoids up to isomorphism are widespread in factorization theory.  The most famous conjecture is stated below for certain block monoids (that~is, monoids of zero-sum sequences of abelian group elements under concatenation).  Since a detailed introduction to block monoids would take us too far afield, we refer the interested reader to~\cite{nonuniq} for background on block monoids and~\cite{lengthsetconj1,lengthsetconj2} for recent progress on Conjecture~\ref{c:blocklensets}.  In what follows, let $\mathcal B(G)$ denote the block monoid over a group $G$; $\mathcal L(\mathcal B(G))$ is defined analogously for this class of monoids.  Note that for two abelian groups $G$ and $G'$, $\mathcal B(G) \iso \mathcal B(G')$ if and only if $G \iso G'$.

\begin{conj}\label{c:blocklensets}
Given two finite abelian groups $G$ and $G'$ with $|G|, |G'| > 3$, we have $\mathcal L(\mathcal B(G)) = \mathcal L(\mathcal B(G'))$ implies $\mathcal B(G) \cong \mathcal B(G')$.  
\end{conj}

Theorem~\ref{t:setsoflength} answers an analogous question for numerical monoids in the negative.

\begin{thm}\cite[Corollary~3.4]{setoflengthsets}\label{t:setsoflength}
There exist numerical monoids $S \ne S'$ with $\mathcal L(S) = \mathcal L(S')$.  In particular, sets of length sets do not characterize numerical monoids.  
\end{thm}

\begin{remark}\label{r:arithmetical}
The distinct numerical monoids with equal sets of length sets described in \cite{setoflengthsets} are called~\emph{arithmetical numerical monoids}, which are numerical monoids generated by an arithmetic sequence $a, a+d, \ldots, a+kd$ with $\gcd(a,d) = 1$ and $1 \leq k < a$. Theorem~\ref{t:setsoflength} is proven by computing the set of length sets of arithmetical numerical monoids in terms of the parameters $a,d$ and $k$ and showing that different values of these parameters yield the same set.
\end{remark}

The description of the set of length sets for arithmetical numerical monoids upon which Theorem~\ref{t:setsoflength} is based is rare and, in general, it is difficult to obtain such a description for other families of numerical monoids.  Part of the difficulty comes in computing every factorization length that an element in a numerical monoid can obtain.  Many times, the most interesting factorization lengths for an element are the largest and smallest ones.  The quotient of these two values gives a weaker factorization invariant called \emph{elasticity}, which is defined below.

\begin{defn}\label{d:elasticity}
Let $S$ be a numerical monoid, $n \in S$ nonzero, and $\mathsf L(n)$ its length set.  Let $\mathsf M(n) = \max \mathsf L(n)$ and $\mathsf m(n) = \min \mathsf L(n).$  The \emph{elasticity of $n$} is given by 
$$\rho(n) = \mathsf M(n)/\mathsf m(n).$$
The \emph{set of elasticities of $S$} is given by $R(S) = \{\rho(n) : n \in S, n \neq 0\}$.  The \emph{elasticity of $S$} is given by $\rho(S) = \sup R(S)$.  
\end{defn}

In what follows, we see that $\rho(S) < \infty$ for numerical monoids and that, in fact, $\rho(S) = \max R(S)$.  Moreover, elasticity is one of the few factorization invariants known to admit a closed form for any numerical monoid $S$.  

\begin{thm}[{\cite[Theorem 2.1 \& Corollary 2.3]{elasticity}}]\label{t:holdenmoore}
Let $S = \<n_1, n_2, \ldots, n_k\>$ be minimally generated by $n_1, \ldots, n_k$ with $n_1$ and $n_k$ the smallest and largest generators, respectively.  Then, $\rho(S) = \rho(n_1n_k) = n_k/n_1$ and the set of elasticities $R(S)$ has $n_k/n_1$ as its unique accumulation point.  
\end{thm}

\begin{example}\label{e:elastsets}
Figure~\ref{fig:elastsets} depicts the elasticity of elements in the numerical monoids $S = \<20,21,45\>$ and $S' = \<7,10,13,16\>$.  The maximal values attained are $9/4$ and $16/7$, respectively, as ensured by Theorem~\ref{t:holdenmoore}.  Moreover, the monoid elements with maximal elasticity are precisely the multiples of the least common multiple of the first and last generators.  
\end{example}

\begin{figure}
\includegraphics[width=2.5in]{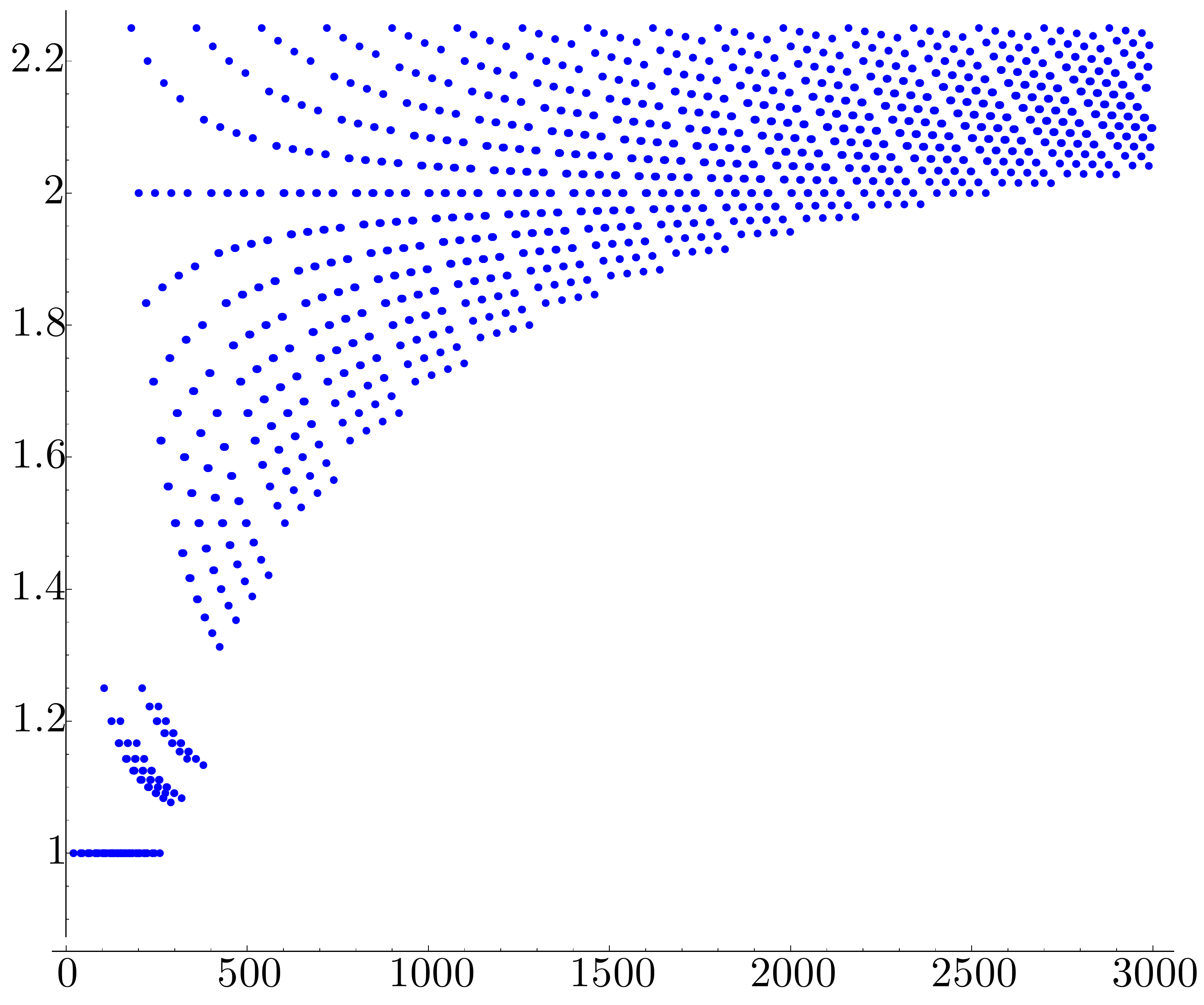}
\hspace{0.5in}
\includegraphics[width=2.5in]{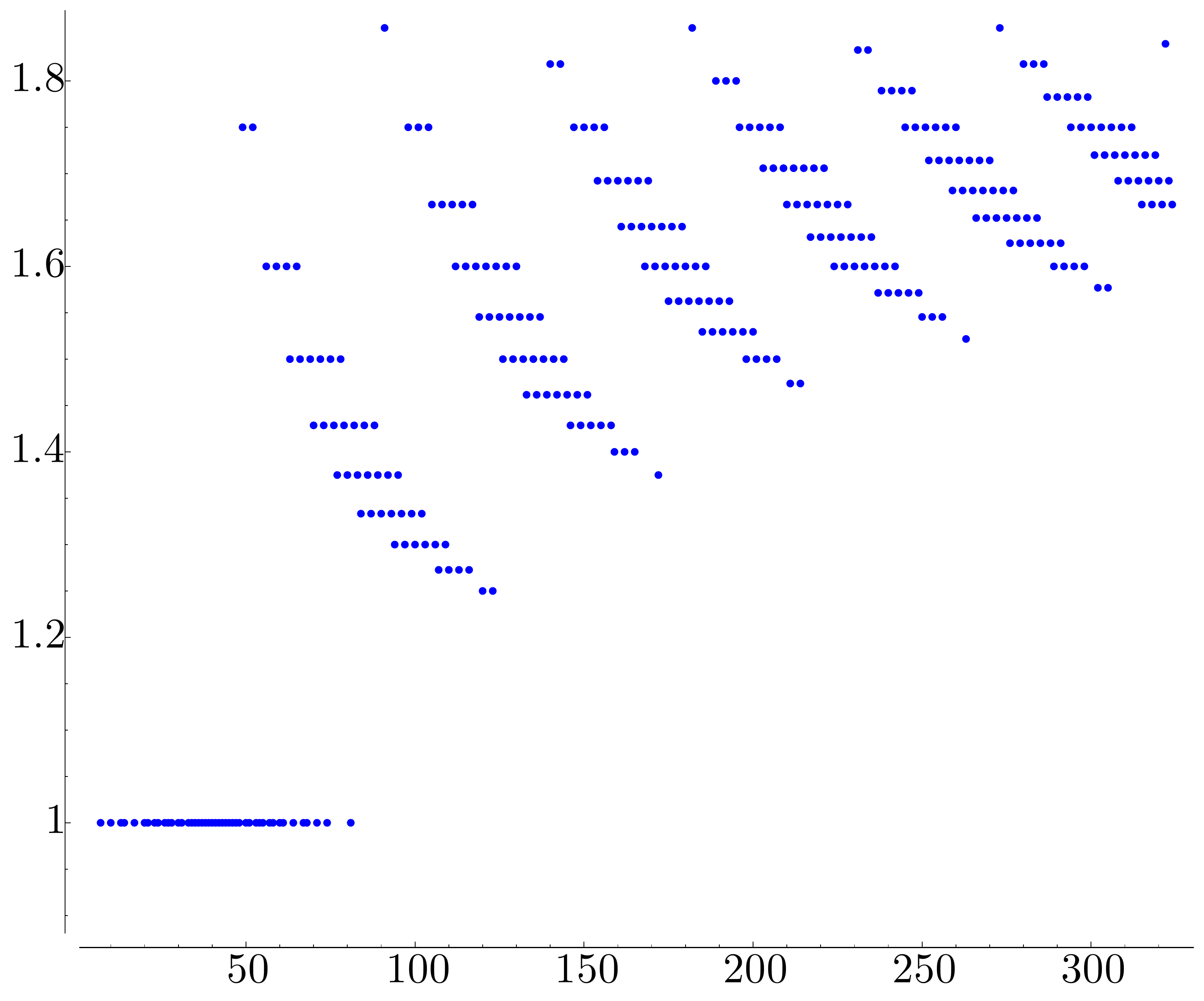}
\medskip
\caption{A plot showing the elasticities for the numerical monoid $\<20,21,45\>$ (left) and the arithmetical numerical monoid $\<7,10,13,16\>$ (right), produced using \texttt{SAGE} \cite{sage}.  Here, the point at $(n,q)$ indicates that $\rho(n) = q$.}  
\label{fig:elastsets}
\end{figure}

More can be said about the structure of the set $R(S)$ upon examining the asymptotic behavior of the maximum- and mininum-length factorization functions $\max \mathsf L_S(n)$ and $\min \mathsf L_S(n)$ for large monoid elements $n \in S$.  See \cite[Section~4]{elastsets} for more detail.  

\begin{thm}{\cite[Corollary~4.5]{elastsets}}\label{t:elasticityset}
Fix a numerical monoid $S = \<n_1, n_2, \ldots, n_k\>$.  
\begin{enumerate}[(a)]
\item For $n \ge n_{k-1}n_k$, we have 
$$\rho(n + n_1n_k) = \frac{M(n) + n_k}{m(n) + n_1}.$$
\item The set $R(S)$ is the union of a finite set and a collection of $n_1n_k$ monotone increasing sequences, each converging to $n_k/n_1$.  
\end{enumerate}
\end{thm}

\begin{remark}\label{r:elasticityset}
Both plots in Figure~\ref{fig:elastsets} demonstrate the claims in Theorem~\ref{t:elasticityset}.  Roughly speaking, each plot can be viewed as a union of ``wedges'' of points, where each wedge contains one point from each of the monotone increasing sequences described in Theorem~\ref{t:elasticityset}(b).  Moreover, the right-hand plot depicts the elasticies of an arithmetical numerical monoid, where elasticity values are often repeated in quick succession.  It is this repetition of elasticitiy values that gives rise to Theorem~\ref{t:elastarith}.  
\end{remark}

Certainly, two numerical monoids with identical sets of length sets have equal sets of elasticities.  Conversely, there exist pairs of numerical monoids with equal elasticity sets and unequal sets of length sets.  For example, for $S = \<6,10,13,14\>$ and $S' = \<6,11,13,14\>$, a computation shows that $R(S) = R(S')$, while $\{4,6\} \in \mathcal L(S)$ but $\{4,6\} \not \in \mathcal L(S')$; see \cite[Example~3.11]{elastsets}.  Thus, in general, sets of elasticities do not determine sets of length sets.  Theorem~\ref{t:elastarith} demonstrates that, within the class of arithmetical numerical monoids, equality of sets of length sets is equivalent to equality of sets of elasticities.

\begin{thm}\cite[Theorem~1.2]{elastsets}\label{t:elastarith}
For two distinct arithmetical numerical monoids $S = \<a, a+d, \ldots, a+kd\>$ and $S' = \<a', a'+d',\ldots,a'+k'd'\>$, the following are equivalent:
\begin{enumerate}[(a)]
\item $R(S) = R(S')$
\item $\mathcal L(S) = \mathcal L(S')$
\end{enumerate}
\end{thm}

Theorem~\ref{t:elasticityset} demonstrates that computing the set of elasticities is fairly tractable, while the set of length sets is, in general, much more difficult to describe.  Therefore, finding classes of numerical monoids (beyond the arithmetical case) for which the set of elasticities is as strong an invariant as the set of length sets is of much interest.  

\begin{prob}\label{pr:elastlencomp}
Characterize which distinct numerical monoids $S$ and $S'$ satisfy $R(S) = R(S')$ and $\mathcal L(S) \ne \mathcal L(S')$.  
\end{prob}

\section{The delta set}\label{s:deltasets}

In this section, we consider the \emph{delta set} invariant (Definition~\ref{d:deltaset}).  The delta set of a monoid element $n$, like its elasticity (Definition~\ref{d:elasticity}), is derived from the set of factorization lengths $\mathsf L(n)$.  

\begin{defn}\label{d:deltaset}
Fix a numerical monoid $S = \<n_1, \ldots, n_k\>$, and fix $n \in S$.  Writing $\mathsf L(n) = \{\ell_1 < \cdots < \ell_r\}$, the \emph{delta set} of $n$ is the set $\Delta_S(n) = \{\ell_i - \ell_{i-1} : 2 \le i \le r\}$ of successive differences of factorization lengths of $n$, and $\Delta(S) = \bigcup_{n \in S} \Delta_S(n)$.  When there is no ambiguity, we often omit the subscript and simply write $\Delta(n)$.  
\end{defn}

\begin{example}\label{e:delta2gen}
Consider the numerical monoid $S = \<n_1, n_2\>$ with two relatively prime generators.  If $n \in S$, then there exist $a, b \in N$ such that $n = an_1 + bn_2$.  Notice that for any $k \in \mathbb Z$, $n = (a + kn_2)n_1 + (b - kn_1)n_2$ is a factorization so long as $a + kn_2$ and $b-kn_1$ are non-negative.  In fact, it can be shown that all factorizations of $n$ in $S$ are of this form, corresponding to successive values of $k$.  Since the length of such a factorization is $a + b + k(n_2 - n_1)$, all successive length differences in $\mathsf L(n)$ yield $n_2 - n_1$.  Thus, for any $n \in S$, either $\Delta(n)$ is empty (when its length set is a singleton), or $\Delta(n) = \{n_2 - n_1\}.$  In particular, 
$$\Delta(S) = \{n_2 - n_1\}.$$
\end{example}

The above example is a special case of arithmetical numerical monoids, whose minimal generators form an arithmetic sequence (Remark~\ref{r:arithmetical}).  Such monoids have singleton delta sets consisting of the step size between these generators, a fact which plays a crutial role in characterizing their length sets (Theorems~\ref{t:setsoflength}).  

\begin{thm}\cite[Corollary~2.3]{setoflengthsets}\label{t:deltaarithmetic}
Fix positive integers $a$, $d$, and $k$ with $\gcd(a,d) = 1$ and $1 \leq k < a$.  For $S = \<a, a + d, a + 2d, \ldots, a + kd\>$ and $n \in S$, either $\Delta(n)$ is empty, or $\Delta(n) = \{d\}$.  In particular, $\Delta(S) = \{d\}$.  
\end{thm}

Even slight generalizations of arithmetical numerical monoids have elusive delta sets; this leads us to the following open problem.  

\begin{prob}\label{pr:genarithmeticdelta}
Let $a$, $h$, $d$, and $k$ be positive integers such that $\gcd(a,d) = 1$ and $k < a$, and consider $S = \<a, ah + d, ah + 2d, \ldots, ah + kd\>$.  Compute $\Delta(S)$.  
\end{prob}

\begin{remark}\label{r:genarithmetic}
The generators of the numerical monoid in Problem~\ref{pr:genarithmeticdelta} form what is known as a \emph{generalized arithmetic sequence}.  Numerical monoids generated by generalized arithmetic sequences and arithmetic numerical monoids are closely connected: when $h = 1$, a generalized arithmetic sequence reduces to an arithmetic sequence; also, the last $k$ generators of a generalized arithmetic sequence form an arithmetic sequence.  These relationships, together with a precise membership criterion~\cite[Theorem~3.1]{omidali}, make working with this family of numerical monoids more tractable.  
\end{remark}

\begin{remark}\label{r:deltadifficulty}
Aside from the arithmetical setting, few natural families of numerical monoids are known to admit closed forms for their delta sets.  In general, given a numerical monoid $S$, it can be hard to prove that $d \notin \Delta(S)$ for a given integer $d$, especially if $\min \Delta(S) < d < \max \Delta(S)$, since it must be shown that $d$ does not lie in the delta set of even a single element of $S$.  Example~\ref{e:deltanumerical} demonstrates the sublety of this question.  
\end{remark}

\begin{example}\label{e:deltanumerical}
The delta set of elements of the numerical monoid $S = \<17,33,53,71\>$ are depicted in Figure~\ref{fig:deltanumerical}.  Notice that $\Delta(S) = \{2,4,6\}$, but $6$ only appears in $\Delta(266)$, $\Delta(283)$, and $\Delta(300)$.  Indeed, for $n > 300$, $\Delta(n) \subset \{2,4\}$; see Theorem~\ref{t:deltaperiodic}.  
\end{example}

\begin{figure}
\includegraphics[width=6in]{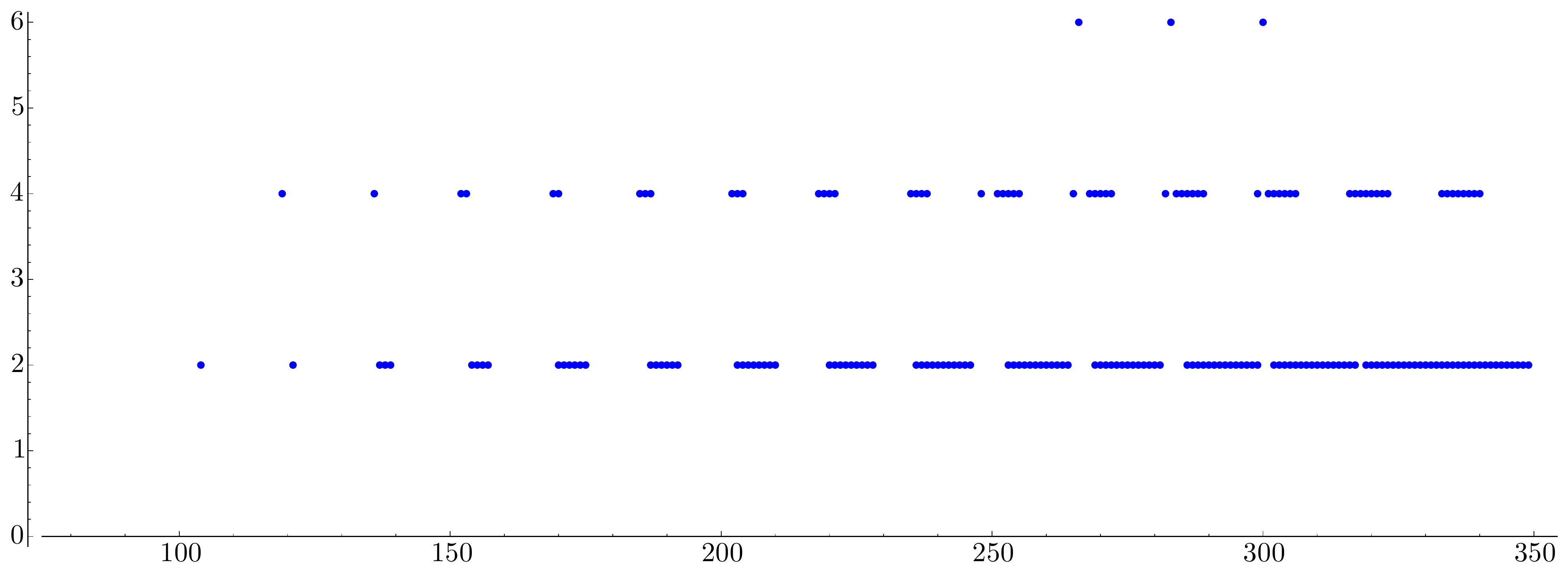}
\medskip
\caption{A plot showing the delta sets of elements in the numerical monoid $\<17,33,53,71\>$ from Example~\ref{e:deltanumerical}, produced using \texttt{SAGE} \cite{sage}.  Here, each point $(n,d)$ indicates that $d \in \Delta(n)$.}  
\label{fig:deltanumerical}
\end{figure}

\begin{remark}\label{r:deltablock}
In contrast to numerical monoids, the delta set of any block monoid $M = \mathcal B(G)$ over a finite abelian group $G$ is known to be an interval, and recent work has focused on determining its maximal element \cite[Section~6.7]{nonuniq}.  In particular, any positive integer $d \le \max \Delta(M)$ must also lie in $\Delta(M)$.  Roughly speaking, block monoids have abundant irreducible elements, so if $d \in \Delta(a)$ for some $a \in M$, one can produce an element in $M$ with $d - 1$ in its delta set by simply adding the appropriate irreducible elements to $a$.  This process effectively ``cuts down'' a length gap $d$ by carefully filling in a particular factorization length.  
\end{remark}

The difficulty in controlling the phenomenon discussed in Remark~\ref{r:deltablock} when considering numerical monoids motivates Problem~\ref{pr:deltarealization}, which is known as the \emph{delta set realization problem for numerical monoids}.  

\begin{prob}\label{pr:deltarealization}
Let $D \subset \ZZ_{>0}$ be a finite set of positive integers.  Determine if there exists a numerical monoid $S$ such that $\Delta(S) = D$.  
\end{prob}

Currently, the only known restrictions on $\Delta(S)$ for a numerical monoid $S$ are given in Theorem~\ref{t:mindelta}, each of which also holds for a much larger class of monoids.  

\begin{thm}[{\cite[Proposition~1.4.4]{nonuniq}}]\label{t:mindelta}
If $M$ is an atomic, cancellative, commutative monoid with $\Delta(M) \ne \emptyset$, then, $\min(\Delta(M)) = \gcd(\Delta(M))$.  If, additionally, $M$ is finitely generated, then $\Delta(M)$ is finite.  
\end{thm}

\begin{remark}\label{r:deltarealization}
Progress on Problem~\ref{pr:deltarealization} has come from finding specific subsets that can be realized as delta sets.  For example, every set of the form $\{d, 2d, \ldots, td\}$ is known to occur as the delta set of a numerical monoid~\cite{delta}.  Additionally, a family of numerical monoids whose delta sets are $\{d, td\}$ for any $t,d > 1$ is given~\cite{deltarealization}.  
\end{remark}

Recent investigation into Problem~\ref{pr:deltarealization} has greatly benefited from evidence produced by computer algebra systems.  That said, until very recently, there was no known algorithm to compute the delta set of an entire numerical monoid $S$.  Although the delta set $\Delta(n)$ of an element $n \in S$ can be computed by computing its set of factorizations $\mathsf Z(n)$, it is less straight forward to compute the entire set $\Delta(S)$, as this requires taking the union of the delta sets of all (infintely many) elements of $S$.  

Before discussing how this issue can be resolved (see Corollary~\ref{c:deltacompute}), we state Theorem~\ref{t:maxdelta}, which demonstrates that the maximal element $\max \Delta(S)$ in the delta set of a numerical monoid $S$ can be obtain by computing the delta sets of a special (computable) class of monoid elements, called \emph{Betti elements} (Definition~\ref{d:betti}).  

\begin{defn}\label{d:betti}
Let $S = \<n_1, n_2, \ldots, n_k\>$ be a numerical monoid and $n \in S$.  Consider the graph $G_n$ with vertex set given by its set of factorizations $\mathsf Z(n)$ and an edge connecting $\aa, \bb \in \mathsf Z(n)$ if $\aa$ and $\bb$ have disjoint support as vectors.  That is, $\aa=(a_1, \ldots, a_k)$ and $\bb = (b_1,\ldots,b_k)$ are adjacent in $G_n$ if for all $i$, $a_i$ and $b_i$ are never both non-zero.  An element $n \in S$ is called a \emph{Betti element} if $G_n$ is disconnected.   The \emph{set of Betti elements of $S$} is denoted by $\Betti(S)$.
\end{defn}

\begin{remark}\label{r:betti}
Every numerical monoid contains only finitely many Betti elements.  When $S = \<n_1,n_2\>$ is a numerical monoid with two irreducible elements, $n_1n_2$ is the unique Betti element.  In the case when $S = \<n_1, n_2, n_3\>$ has three irreducible elements, there are at most three Betti elements, each of which is an integer multiple of one of the generators of $S$ \cite[Section~9.3]{numerical}.  
\end{remark}

\begin{thm}[{\cite[Theorem~2.5]{deltabetti}}]\label{t:maxdelta}
Let $S$ be a numerical monoid.  The largest element of $\Delta(S)$ is achieved at a Betti element of $S$.  In other words, 
$$\max \Delta(S) = \max\{d : d \in \Delta(b) \text{ for } b \in \Betti(S)\}.$$
\end{thm}

Betti elements are discussed in more depth in Section~\ref{s:catenary}, as they are useful in finding maximal and minimal catenary degrees achieved in a numerical monoid (see Definitions~\ref{d:catenarydegree} and~\ref{d:catenaryset}).  

\begin{example}\label{e:bettidelta}
As stated in Remark~\ref{r:betti}, for the numerical monoid $S = \<n_1, n_2\>$ with two irreducible elements, $n_1n_2$ is the unique Betti element.  In fact, it has precisely two factorizations, namely $(n_2,0)$ and $(0,n_1)$.  Since their difference in length equals $n_2 - n_1$, we have $\max \Delta(S) = n_2 - n_1$.  Of course, by Theorem~\ref{t:deltaarithmetic}, $\Delta(S) = \{n_2 - n_1\}.$  
\end{example}

Theorem~\ref{t:deltaperiodic} states that delta sets for numerical monoids are eventually periodic with period dividing $\lcm(n_1,n_k)$.  The value of $N_S$ given in~\cite{compasympdelta} is explicitly computed, albeit cumbersome to state.  Most impressive about Theorem~\ref{t:deltaperiodic} is that it provides a concrete upper bound for the beginning of the periodic behavior, and thus a finite set of monoid elements the union of whose delta sets equals the delta set of the entire numerical monoid.  As a consequence, Theorem~\ref{t:deltaperiodic} also yields the first known algorithm to compute the delta set of any numerical monoid; see Corollary~\ref{c:deltacompute}.  

\begin{thm}[{\cite[Theorem~1]{compasympdelta}}]\label{t:deltaperiodic}
For $S = \<n_1, \ldots, n_k\>$ numerical, there exists a number $N_S$ such that for any $n \in S$ with $n \geq N_S$, $\Delta(n) = \Delta(n + \lcm(n_1,n_k))$.  
\end{thm}

\begin{cor}\label{c:deltacompute}
Given a numerical monoid $S = \<n_1, n_2, \ldots, n_k\>$, we have
$$\Delta(S) = \!\!\! \bigcup_{n \in S \cap [0,N]} \!\!\! \Delta(n),$$
where $N = N_S + \lcm(n_1,n_k)$.  
\end{cor}

\begin{example}\label{e:deltaperiodic}
Let $S = \<17,33,53,71\>$ as in Example~\ref{e:deltanumerical}.  Readily visible in Figure~\ref{fig:deltanumerical} is the eventual periodicity of $\Delta:S \to 2^\NN$ described by Theorem~\ref{t:deltaperiodic}.  In particular, Theorem~\ref{t:deltaperiodic} ensures that $\Delta(n) = \Delta(n + \lcm(17,71))$ for all $n \ge N_S = 6461$.  Computing $\Delta(n)$ for smaller values of $n$ demonstrates that, in~fact, $\Delta(n) = \Delta(n + 17)$ holds for all $n \ge 319$.  
\end{example}

\begin{remark}\label{r:deltacompute}
While Corollary~\ref{c:deltacompute} provides the delta set of the entire numerical monoid by taking the union of only finite many element-wise delta sets, the number of elements $n$ for which $\Delta(n)$ needs to be computed becomes prohibitively large, especially when larger generators are used.  Aside from recent improvements in the special case when $S$ has 3 minimal generators \cite{deltadim3}, few refinements are known for the set of elements whose delta sets must be computed.  Alternatively, we can utilize a dynamic algorithm for computing length sets to dynamically compute $\Delta(n)$ and $\Delta(S)$; see Remark~\ref{r:deltadynamic}.
\end{remark}

\begin{remark}\label{r:deltadynamic}
In the numerical monoid $S = \<6,9,20\>$, it is possible to compute $\mathsf L(60)$ directly from $\mathsf L(40)$, $\mathsf L(51)$, and $\mathsf L(54)$, without computing $\mathsf Z(60)$.  The key observation is that one can produce a factorization of $60$ from a factorization of $40$, $51$ or $54$ by adding a single irreducible element, and this opereration always increases factorization length by exactly one.  Consider the following computation:
$$\begin{array}{rcl}
\mathsf L(60) &=& \left(\mathsf L(40) + 1\right) \,\cup\, \left(\mathsf L(51) + 1\right) \,\cup\, \left(\mathsf L(54) + 1\right) \\
&=& \{3\} \,\cup\, \{7,8,9\} \,\cup\, \{7,8,9,10\} \\
&=& \{3,7,8,9,10\}.
\end{array}$$
This provides an inductive (dynamic) method to compute all the length sets necessary to employ Corollary~\ref{c:deltacompute} without having to compute any factorization sets, and the resulting algorithm is significantly faster than one dependent on factorizations.  Indeed, the set $\mathsf Z(n)$ of factorizations of $n$ grow large very quickly~\cite{factorasymp}, whereas $|\mathsf L(n)|$ grows linearly in $n$ \cite[Theorem~4.3]{elastsets}.  For a more detailed analysis, see~\cite[Section~3]{dynamicalg}.  
\end{remark}

\section{The catenary degree}\label{s:catenary}

Many of the factorization invariants discussed above (e.g., length sets, delta sets, and elasticity) measure only information about the lengths of factorizations of an element.  While helpful in measuring the non-uniqueness of factorizations of an element, these invariants lose certain information when passing from the set of factorizations to the set of lengths (Theorem~\ref{t:setsoflength}).  The \emph{catenary degree}, which we describe in this section, is derived directly from factorizations and is not uniquely determined by sets of factorization lengths.  

\begin{defn}\label{d:catenarydistance}
Fix a numerical monoid $S = \<n_1, \ldots, n_k\>$, and fix an element $n \in S$.  For $\aa, \bb \in \mathsf Z(n)$, the \emph{greatest common divisor of $\aa$ and $\bb$} is given by 
$$\gcd(\aa,\bb) = (\min(a_1,b_1), \ldots, \min(a_r,b_r)) \in \NN^r,$$
and the \emph{distance between $\aa$ and $\bb$} (or the \emph{weight of $(\aa,\bb)$}) is given by 
$$d(\aa,\bb) = \max(|\aa - \gcd(\aa,\bb)|,|\bb - \gcd(\aa,\bb)|).$$
\end{defn}

The function $d: \mathsf Z(n) \times \mathsf Z(n) \to \mathbb N$ measures how many irreducible elements two factorizations do not share in common.  As implied by the notation, $d(-,-)$ is a metric on the set $\mathsf Z(n)$ \cite[Proposition~1.2.5.3]{nonuniq}.  In fact, it is common to represent $\mathsf Z(n)$ as the vertices of a complete graph, where the edge between factorizations $\aa$ and $\bb$ is decorated with $d(\aa,\bb)$.  

Below, we use the metric $d$ to define the \emph{catenary degree}, which measures the ``overall'' distance between distinct factorizations of a given monoid element.  

\begin{defn}\label{d:catenarydegree}
Given $\aa, \bb \in \mathsf Z(n)$ and $N \ge 1$, an \emph{$N$-chain from $\aa$ to $\bb$} is a sequence $\aa_1, \ldots, \aa_r \in \mathsf Z(n)$ of factorizations of $n$ such that (i) $\aa_1 = \aa$, (ii) $\aa_r = \bb$, and (iii) $d(\aa_{i-1},\aa_i) \le N$ for all $i \le r$.  The \emph{catenary degree of $n$}, denoted $\mathsf c(n)$, is the smallest $N \in \NN$ such that there exists an $N$-chain between any two factorizations of $n$.  
\end{defn}

With respect to the decorated complete graph representing $\mathsf Z(n)$, the catenary degree $\mathsf c(n)$ can be computed by repeatedly removing edges of maximal distance until the graph becomes disconnected.  The distance decorating the last removed edge is equal to $\mathsf c(n)$.  Example~\ref{e:catenarydegree} describes such a computation.

\begin{example}\label{e:catenarydegree}
Consider the numerical monoid $S = \<11,36,39\>$.  The left-hand picture in Figure~\ref{fig:catenarydegree} depicts the factorizations of $450 \in S$ along with all pairwise distances.  There exists a $16$-chain between any two factorizations of $450$; one such $16$-chain between $(6,2,8)$ and $(24,3,2)$ is depicted with bold red edges.  Since every $16$-chain between these factorizations contains the edge labeled 16, we have $\mathsf c(450) = 16$.  

This can also be computed in a different way.  In the right-hand picture of Figure~\ref{fig:catenarydegree}, only distances of at most $16$ are depicted, and the resulting graph is connected.  Removing the edge labeled $16$ yields a disconnected graph, so we again conclude $\mathsf c(450) = 16$.  
\end{example}

\begin{figure}
\includegraphics[width=2.5in]{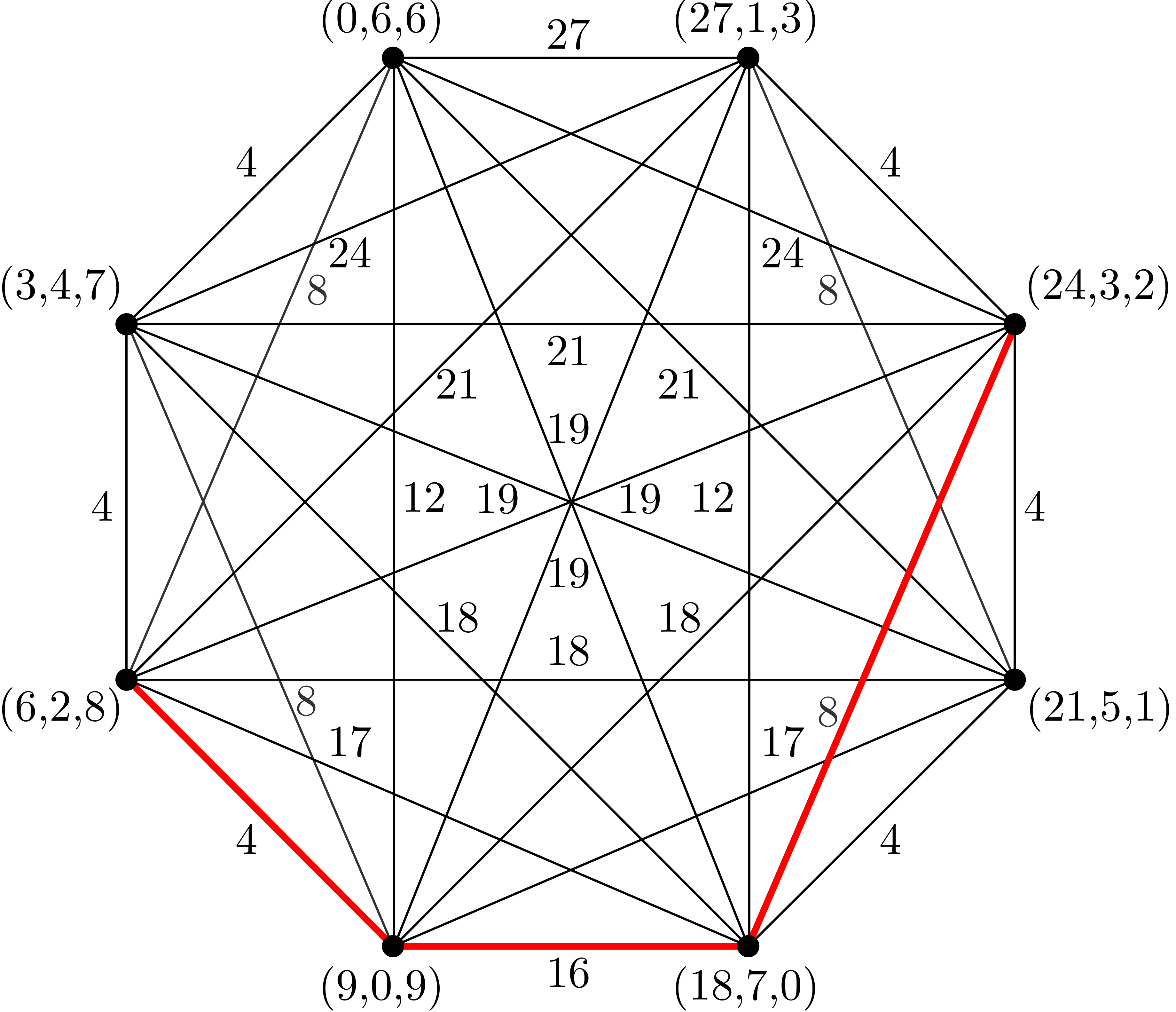}
\hspace{0.5in}
\includegraphics[width=2.5in]{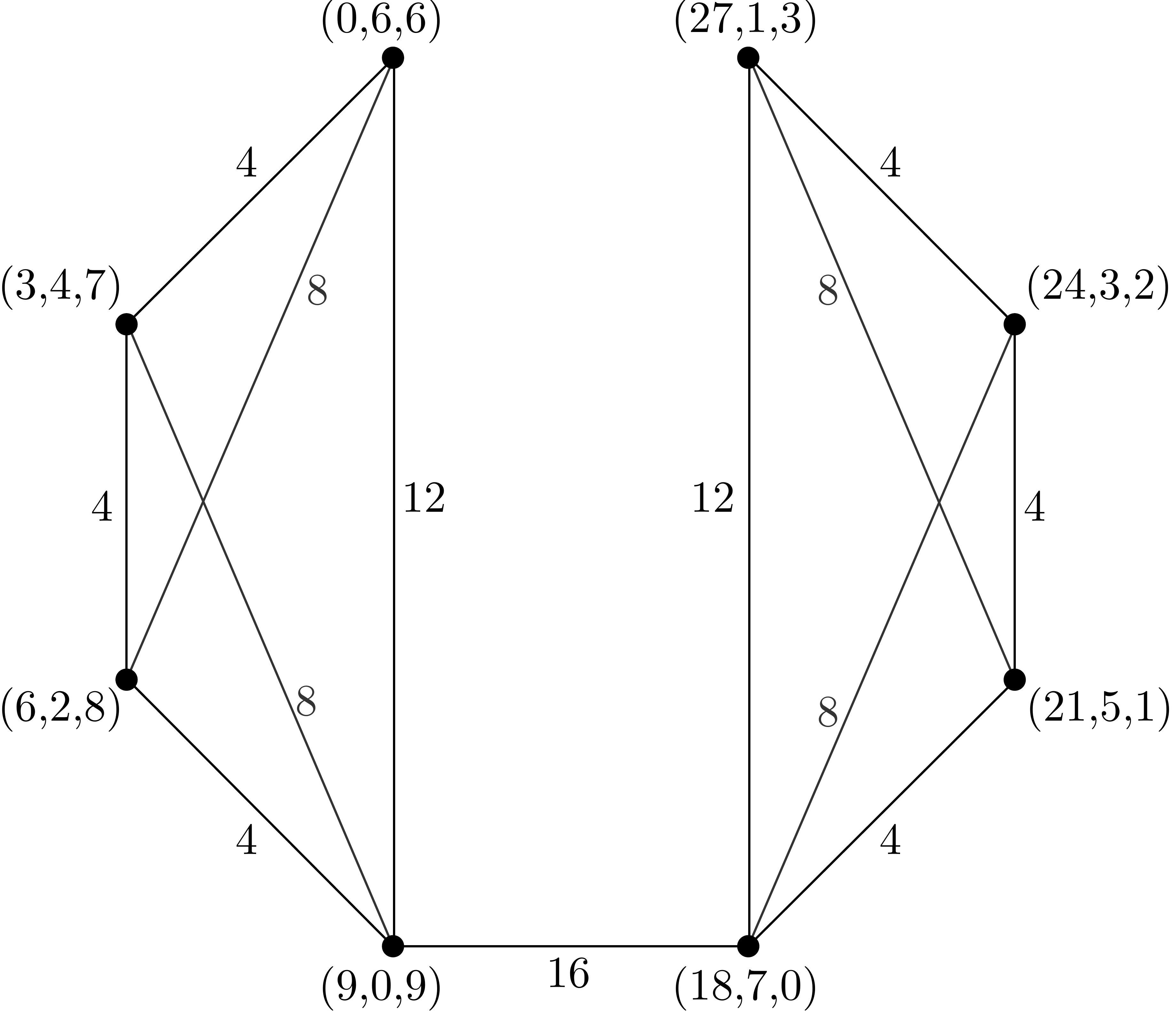}
\medskip
\caption{Computing the catenary degree of $450 \in S = \<11,36,39\>$, as in Example~\ref{e:catenarydegree}.  Each vertex is labeled with an element of $\mathsf Z(450)$, and each edge is labeled with the distance between the factorizations at either end.  The diagram on the left depicts all edges, and the diagram on the right includes only those edges labeled at most $\mathsf c(450) = 16$.  Both graphics were created using the computer algebra system \texttt{SAGE} \cite{sage}.}  
\label{fig:catenarydegree}
\end{figure}

While the catenary degree of an individual element contains information about its different factorizations, the range of catenary degrees achieved by all elements in a monoid has also been studied.  Definition~\ref{d:catenaryset} gives two ways to measure how the catenary degree varies throughout a numerical monoid.

\begin{defn}\label{d:catenaryset}
The \emph{set of catenary degrees of $S$} is the set $\mathsf C(S) = \{\mathsf c(m) : m \in S\}$, and the  \emph{catenary degree of $S$} is the maximal catenary degree: $\mathsf c(S) = \max \mathsf C(S)$.
\end{defn}

By far, the more well-understood of these two invariants is $\mathsf c(S)$, which has been computed for several classes of numerical monoids.  Theorem~\ref{t:catenarycomputation} gives a summary of some known results (additionally, see Theorem~\ref{t:maxcatdegree}).  

\begin{thm}[{\cite{compseqs,omidali}}]\label{t:catenarycomputation}
Fix a numerical monoid $S$.  
\begin{enumerate}[(a)]
\item If $S = \<n_1, n_2\>$, then $\mathsf c(S) = n_2$.  
\item If $S = \<a, ah + d, \ldots, ah + kd\>$ for positive integers $a, h, k, d$ with $\gcd(a,d) = 1$ and $k < a$, then $\mathsf c(S) = \lceil \frac{a}{k} \rceil h + d$.  
\item If $S = \<(\prod_{j=1}^i b_j)(\prod_{j=i+1}^k a_j) : 0 \le i \le k\>$, where $2 \le a_i < b_i$ and $\gcd(a_i,b_j) = 1$ for all $i \ne j \le k$, then $\mathsf c(S) = \max\{b_1, \ldots, b_k\}$.  
\end{enumerate}
\end{thm}

While a closed form for $\mathsf c(n)$ for every element $n$ is a numerical monoid is rare, such a formula has been found for numerical monoids generated by an arithmetic sequence.

\begin{thm}[{\cite[Theorem~3.1]{catenaryarithseq}}]\label{t:catnaryarithmetic}
Suppose $S = \<a, a + d, \ldots, a + kd\>$ for positive integers $a, k, d$ with $\gcd(a,d) = 1$ and $k < a$.  Then 
$$\mathsf c(n) = \left\{\begin{array}{ll}
0 & |\mathsf Z(n)| = 1 \\
2 & |\mathsf L(n)| = 1 \text{ and } |\mathsf Z(n)| > 1 \\
\mathsf c(S) & |\mathsf L(n)| > 1
\end{array}\right.$$
for each nonzero $n \in S$, where $\mathsf c(S) = \lceil \frac{a}{k} \rceil + d$.  In particular, $\mathsf C(S) = \{0,2,\mathsf c(S)\}$.  
\end{thm}

While it remains a lofty goal to give a closed form for the catenary degree of individual elements in the general setting, the set of catenary degrees may prove to be more approachable.  Problem~\ref{pr:catenarysets} is, in many ways, an analog of Problem~\ref{pr:deltarealization} for the set of catenary degrees.  

\begin{prob}\label{pr:catenarysets}
Determine which finite sets equal $\mathsf C(S)$ for some numerical monoid~$S$.  
\end{prob}

Many of the arguments that compute the maximal catenary degree $\mathsf c(S)$ for a finitely generated monoid $S$ focus on the Betti elements of $S$ (Definition~\ref{d:betti}).  The factorizations of Betti elements contain enough information about $\mathsf Z(S)$ to give sharp bounds on the catenary degrees occuring in $\mathsf C(S)$ (Theorem~\ref{t:maxcatdegree}), and can be readily computed using the \texttt{GAP} package \texttt{numericalsgps} \cite{numericalsgpsgap}.  

\begin{thm}[{\cite{catenarytamefingen,mincatdeg}}]\label{t:maxcatdegree}
For any finitely generated monoid $S$, we have 
$$\begin{array}{r@{\,\,}c@{\,\,}l}
\mathsf c(S) = \max \mathsf C(S) &=& \max\{\mathsf c(b) : b \in \Betti(S)\} \text{ and } \\
\min \mathsf C(S) &=& \min\{\mathsf c(b) : b \in \Betti(S)\}.  
\end{array}$$
\end{thm}

\begin{example}\label{e:bettielements}
Let $S = \<11,25,29\>$.  The catenary degrees of $S$ are plotted in Figure~\ref{fig:bettielements}.  The only Betti elements of $S$ are 58, 150, and 154, which have catenary degrees 4, 12, and 14, respectively.  The Betti elements can be seen in the plot as the first occurence of each of these three values.  Notice that $\mathsf c(175) = 11$ is distinct from each of these values.  By Theorem~\ref{t:maxcatdegree}, every element of $S$ with at least two distinct factorizations has catenary degree at least 4 and at most 14.  
\end{example}

\begin{figure}
\includegraphics[width=5in]{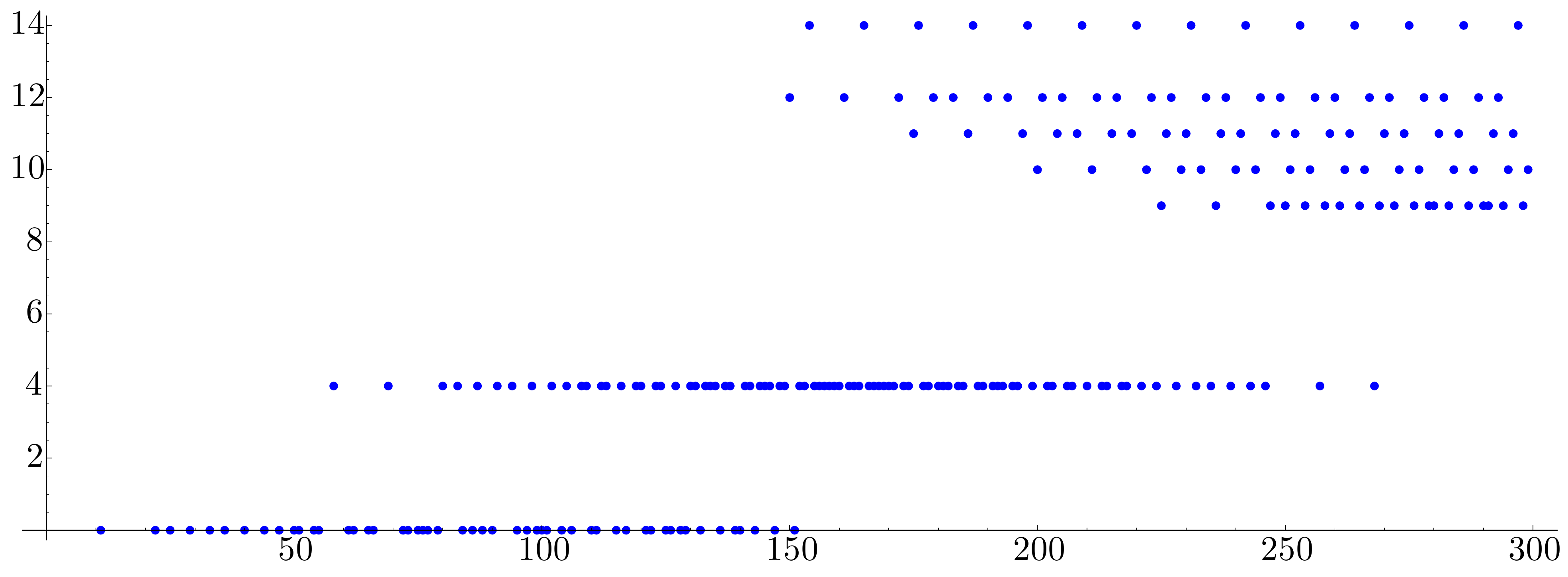}
\medskip
\caption{A \texttt{SAGE} plot \cite{sage} showing the catenary degrees for the numerical monoid $S = \<11,25,29\>$ discussed in Example~\ref{e:bettielements}.}
\label{fig:bettielements}
\end{figure}

As with the delta set, understanding the long-term behavior of the catenary degree (when viewed as a function $\mathsf c: S \to \mathbb N$) is of particular interest.  Theorem~\ref{t:catenaryperiodic} shows that the catenary degree in numerical monoids is also eventually periodic, a phenomenon that can be seen in Figure~\ref{fig:bettielements}.  

\begin{thm}[{\cite[Theorem~3.1]{catenaryperiodic}}]\label{t:catenaryperiodic}
Fix a numerical monoid $S = \<n_1, \ldots, n_k\>$.  The catenary degree function $\mathsf c:S \to \NN$ is eventually periodic, and moreover, its period divides $\lcm(n_1, \ldots, n_k)$.  
\end{thm}

\begin{remark}\label{r:catenaryperiodic}
Among the eventual periodicity results in this paper (Theorems~\ref{t:elasticityset}, \ref{t:deltaperiodic}, \ref{t:catenaryperiodic}, and~\ref{t:omegaquasi}), Theorem~\ref{t:catenaryperiodic} provides the most room for improvement.  There is currently no known bound on the start of this periodic behavior, as its proof in \cite{catenaryperiodic} relies on the fact that a nonincreasing sequence of positive integers is eventually constant.  Additionally, the period is only known to divide the least common multiple of the monoid generators, but in practice this bound is far from sharp.  It remains an interesting problem to refine understanding of these properties.  
\end{remark}

Examination of the delta set and the $\omega$-primality invariant (intoduced in Section~\ref{s:omega}) both benefit from the use of dynamic algorithms; see Remarks~\ref{r:deltadynamic} and~\ref{r:omegadynamic}.  An answer to Problem~\ref{pr:catenarydynamic} is likely within reach, and would aide in refining Theorem~\ref{t:catenaryperiodic}.  

\begin{prob}\label{pr:catenarydynamic}
Find a dynamic algorithm to compute the catenary degree of numerical monoid elements.  
\end{prob}

\section{$\omega$-primality}\label{s:omega}

In cancellative, commutative monoids, non-unique factorizations arise from irreducible elements that are not prime.  The rich factorization theory in numerical monoids stems in part from the fact that no element is prime.  In this section, we develop the notion of $\omega$-primality, which measures how ``far from prime'' an element is in a natural way (see \cite{prime,omegamonthly} for a more thorough introduction).  We begin with a definition.   

\begin{defn}\label{d:omega}
Let $S = \<n_1, n_2, \ldots, n_k\>$ be a numerical monoid.  For $n \in S$, define the \emph{$\omega$-primality of $n$} to be $\omega(n) = m$ if $m$ is the smallest positive integer with the property that whenever $\aa \in \NN^k$ satisfies $\sum_{i=1}^k a_in_i - n \in S$ with $|\aa| > m$, there exists a $\bb \in \NN^k$ with $b_i \leq a_i$ for each $i$ such that $\sum_{i=1}^k b_i n_i - n \in S$ and $|\bb| \leq m$.  
\end{defn}

In a multiplicative monoid $M$, an element $n \in M$ is prime if whenever $n \mid ab$ for $a,b \in M$, then either $n \mid a$ or $n \mid b$.  Definition~\ref{d:omega} is a generalization of primeness, written additively in the context of a numerical monoid.  In fact, in a general cancellative, commutative monoid, $\omega(n) = 1$ if and only if $n$ is prime.  

In practice, the $\omega$-primality of an element $n$ of a numerical monoid $S = \<n_1, \ldots, n_k\>$ is computed by finding those $\aa \in \NN^k$ that are maximal among those with the property that $\sum_{i=1}^k a_in_i - n$ lies in $S$ (Definition~\ref{d:bullet}).  Proposition~\ref{p:bullet} states their precise relationship to the $\omega$-value of $n$.  

\begin{defn}\label{d:bullet}
We say that $\aa \in \NN^k$ is a \emph{bullet} for $n$ if (i)~$\left(\sum_{i=1}^n a_i n_i\right) - n \in S$ and (ii)~$(\sum_{i=1}^k a_in_i) - n_j - n \not \in S$ whenever $a_j > 0$.  A bullet $\aa$ for $n$ is \emph{maximal} if $|\vec a| = \omega(n)$.  The \emph{set of bullets} for an element $n$ is denoted by $\bul(n)$.
\end{defn}

In general, an element $n$ has several bullets with different lengths.  As stated in Proposition~\ref{p:bullet}, the largest such length coincides with the $\omega(n)$, justifying the above definition for a \emph{maximal bullet}.

\begin{prop}[{\cite[Proposition~2.10]{omegamonthly}}]\label{p:bullet}
Given any numerical monoid $S$, 
$$\omega(n) = \max\{|\aa| : \aa \in \bul(n)\}$$
for each element $n \in S$.  
\end{prop}

\begin{example}\label{e:omega}
Let $S = \<6,9,20\>$ denote the numerical monoid from Example~\ref{e:numericalmonoid}.  Consider the following sets: 
$$\begin{array}{rcl}
\bul(40) &=& \{(0,0,2),(4,4,0),(7,2,0),(10,0,0),(1,6,0),(0,8,0)\}, \\
\bul(51) &=& \{(0,7,0),(10,0,0),(4,3,0),(1,5,0),(0,0,3),(7,1,0)\}, \\
\bul(54) &=& \{(9,0,0),(6,2,0),(0,6,0),(3,4,0),(0,0,3)\}, \\
\bul(60) &=& \{(4,4,0),(7,2,0),(10,0,0),(1,6,0),(0,8,0),(0,0,3)\}.
\end{array}$$
We see that $(10,0,0) \in \bul(40)$ since $10 \cdot 6 - 40 = 20 \in S$, but $9 \cdot 6 - 40 = 14 \notin S$.  Similarly, $(7,2,0) \in \bul(60)$ since $7 \cdot 6 + 2 \cdot 9 - 60 = 0 \in S$, but omitting a single $6$ or $9$ produces a difference lying outside of $S$.  Once the above sets have been determined, Proposition~\ref{p:bullet} implies that $\omega(54) = 9$ and $\omega(40) = \omega(51) = \omega(60) = 10$.  
\end{example}

\begin{remark}\label{r:omegadynamic}
Several recent works examine methods of computing $\omega$-primality in numerical monoids \cite{andalg,progomeganumerical,semitheor,omegaasymp}.  At the time of writing, the fastest known algorithm (which is implemented in the \texttt{GAP} package \texttt{numericalsgps} \cite{numericalsgpsgap}) utilizes dynamic programming to compute bullets \cite{dynamicalg}, similar to the dynamic algorithm used to compute the delta set (Remark~\ref{r:deltadynamic}).  Example~\ref{e:omegadynamic} demonstrates the idea behind this algorithm.  
\end{remark}

\begin{example}\label{e:omegadynamic}
Resuming notation from Example~\ref{e:omega}, for each bullet $\aa \in \bul(54)$, either $\aa \in \bul(60)$ or $\aa + \ee_1 \in \bul(60)$.  Notice that it is impossible for both of these to lie in $\bul(60)$ by Definition~\ref{d:bullet}(ii).  Similarly, bullets in $\bul(51)$ and $\bul(40)$ give rise to bullets in $\bul(60)$.  Moreover, each bullet for $60$ is the ``image'' of a bullet for $54$, $51$ or $40$ in this way.  Much like the recurrence relation for sets of factorizations (Remark~\ref{r:deltadynamic}), the set of bullets of an element $n \in S$ is determined by the ``images'' of the bullets of $n - 6$, $n - 9$ and $n - 20$; see \cite{dynamicalg} for more detail.  
\end{example}

In general, the existence of non-unique factorizations coincides with the existence of non-prime irreducible elements.  As such, the $\omega$-primality invariant was originally constructed to measure how far irreducible elements were from being prime (see, for example,~\cite{prime, bounding}).  Although some progress has been made in the setting of numerical monoids \cite{andalg,omegadim3,interval}, Problem~\ref{pr:omegagens} remains largely open.  

\begin{prob}\label{pr:omegagens}
Determine the $\omega$-values of numerical monoid generators.  
\end{prob}

\begin{remark}\label{r:tamedegree}
Both the catenary degree (Definition~\ref{d:catenarydegree}) and a variation called the tame degree share a surprising connection to $\omega$-primality.  Given $n \in S$, the \emph{tame degree} $\mathsf t(n)$ equals the minimum value $m$ such that every factorization of $n$ with at least one zero entry is distance at most $m$ (in the sense of Definition~\ref{d:catenarydistance}) from a fully supported factorization of $n$, or $\mathsf t(n) = 0$ if $n$ has no fully supported factorizations \cite{catenarytamenumerical,catenarytamefingen}.  As with the catenary degree, the tame degree of $S$ is defined as $\mathsf t(S) = \sup \{\mathsf t(n) : n \in S\}$ and is known to be finite.  It is not hard to show that $\mathsf c(S) \le \mathsf t(S)$, but what is perhaps more surprising is that $\mathsf c(S) \le \omega(S) \le \mathsf t(S)$, where $\omega(S) = \max\{\omega(n_1), \ldots, \omega(n_k)\}$ is the maximum $\omega$-value obtained at an irreducible \cite{semitheor}.  This connection has been explored in some of the recent literature, in part as a method of bounding $\omega(S)$ \cite{semitheor,halffactorial,cattame3}.  
\end{remark}

Recent investigations focusing on the $\omega$-values of all non-unit elements in monoids have uncovered interesting asymptotic behavior.  In the setting of numerical monoids, $\omega$-primality grows in a periodic, linear fashion when viewed as a function $\omega: S \to \NN$.

\begin{thm}[{\cite{dynamicalg,omegaasymp,omegaquasi}}]\label{t:omegaquasi}
Let $S = \<n_1, n_2, \ldots, n_k\>$ be a numerical monoid.  The $\omega$-primality function is eventually quasilinear.  More specifically, for $N_0 = \frac{F(S) + n_2}{n_2/n_1 - 1}$, we have $\omega(n) = \frac{1}{n_1}n + a_0(n)$ for every $n > N_0$, where $a_0: \NN \to \QQ$ is $n_1$-periodic.  
\end{thm}

Theorem~\ref{t:omegaquasi} implies that the graph of the function $\omega:S \to \NN$ eventually has the form of $n_1$ lines with common slope $1/n_1$.  Figure~\ref{fig:omegaquasi} demonstrates this phenomenon.  

\begin{figure}
\includegraphics[width=2.5in]{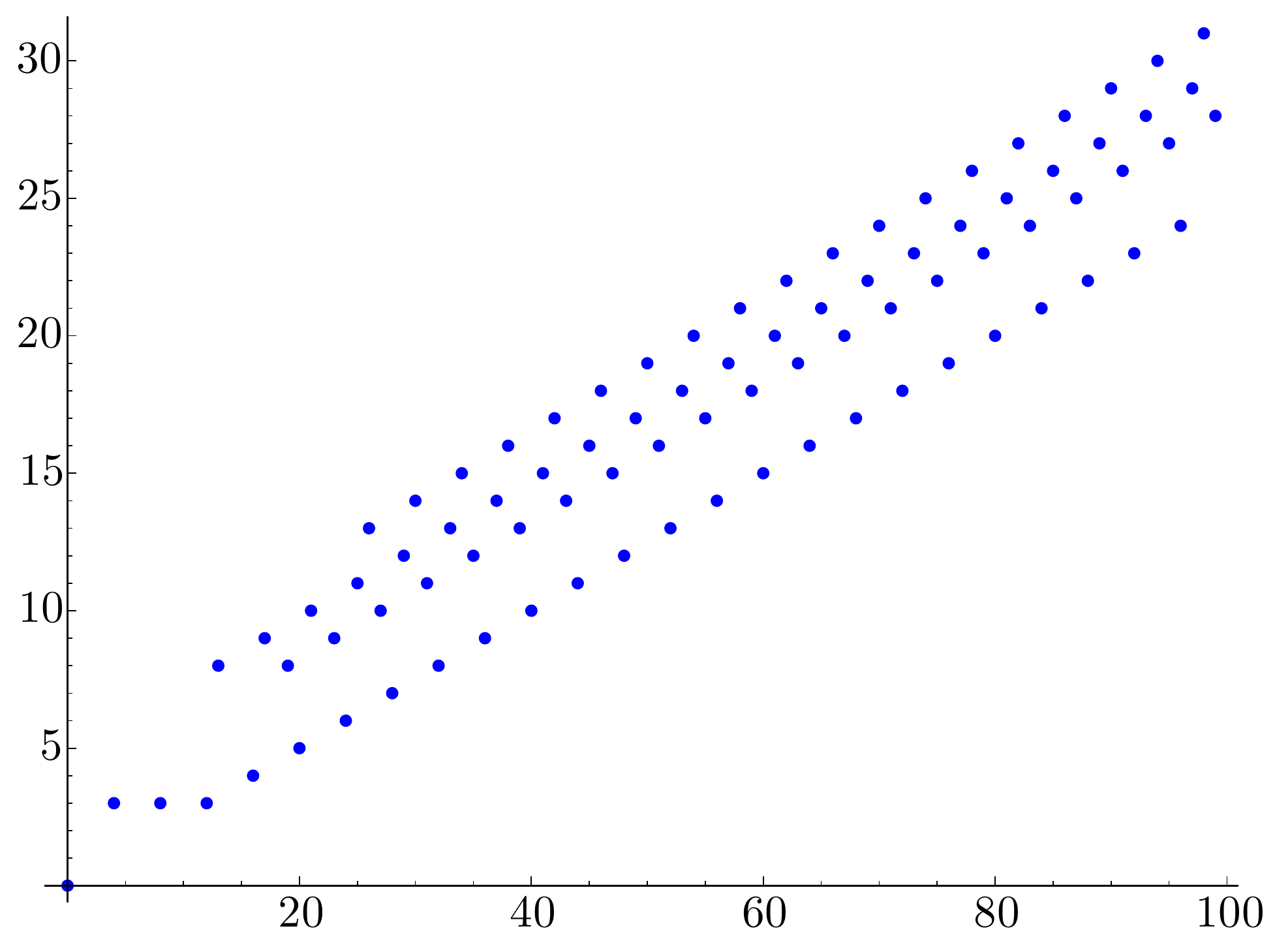}
\hspace{0.5in}
\includegraphics[width=2.5in]{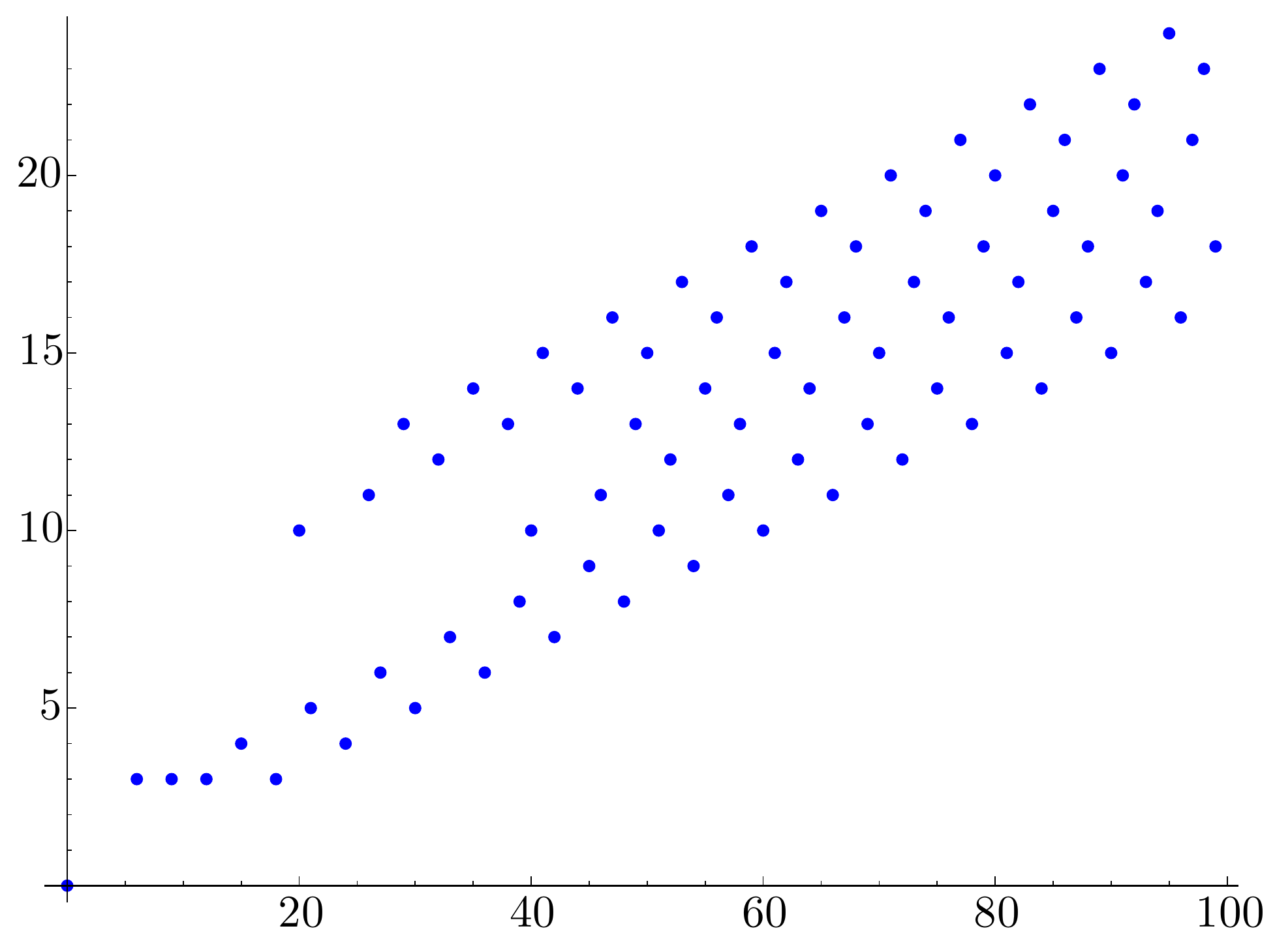}
\medskip
\caption{A plot showing the $\omega$-function for the numerical monoids $\<4,13,19\>$ (left) and $\<6,9,20\>$ (right), produced using \texttt{SAGE} \cite{sage}.}  
\label{fig:omegaquasi}
\end{figure}

\begin{remark}\label{r:omegaquasi}
We resume notation from Theorem~\ref{t:omegaquasi}.  In contrast to Theorem~\ref{t:deltaperiodic}, the period of the function $a_0$ is known to be exactly $n_1$ \cite{dynamicalg}.  Additionally, the value of $N_0$, though not tight, is in general a relatively good bound on the start of the quasilinear behavior of $\omega_S$ \cite[Remark~6.8]{dynamicalg}.  Each of these improvements on the original statement of Theorem~\ref{t:omegaquasi} in \cite{omegaasymp,omegaquasi} are heavily motivated by data produced by the dynamic algorithm described in Remark~\ref{r:omegadynamic}.  
\end{remark}


\end{document}